\documentclass[a4paper,10pt,twoside]{article}

\usepackage{geometry}
\usepackage[T1]{fontenc}
\usepackage[utf8]{inputenc}
\usepackage{bera}

\usepackage{amsmath,amsfonts,amssymb,amsthm,mathtools}

\newtheorem{theorem}{Theorem}[section]
\newtheorem{proposition}[theorem]{Proposition}
\newtheorem{lemma}[theorem]{Lemma}

\usepackage{enumitem}
\usepackage{stmaryrd}
\usepackage[euler]{textgreek}

\newcommand{\Nb}{\mathbb N}
\newcommand{\Rb}{\mathbb R}
\newcommand{\Ac}{\mathcal A}
\newcommand{\Ec}{\mathcal E}
\newcommand{\set}[2]{\left\{ #1 \mathrel{}\middle|\mathrel{} #2 \right\}}
\newcommand{\ds}{\displaystyle}
\newcommand{\indic}[1]{\mathbf{1}_{#1}}			
\newcommand{\prob}[1]{\mathbf{P}\left\{#1\right\}} 	
\newcommand{\probc}[2]{\mathbf{P}\set{#1}{#2}} 	
\newcommand{\probcin}[3]{\mathbf{P}_{#1}\set{#2}{#3}} 	
\newcommand{\probin}[2]{\mathbf{P}_{#1}\left\{#2\right\}} 	
\newcommand{\espin}[2]{\mathbf{E}_{#1}\left[#2\right]}		
\newcommand{\Brak}[2]{\left\langle #1 , #2 \right\rangle}
\newcommand{\init}{{\mathrm in}}
\newcommand{\qsd}{{\mathrm qsd}}
\newcommand{\eq}{{\mathrm eq}}
\newcommand{\stat}{{\mathrm stat}}

\usepackage{titlesec}
\titleformat*{\section}{\large \bfseries}

\numberwithin{equation}{section}

\usepackage{fancyhdr}
\title{Quasi-stationary distribution and metastability for the stochastic Becker-D\"oring model} 

\author{Erwan~Hingant\thanks{Departamento de Matemática, Universidad del Bío-Bío, Chile. E-mail: \texttt{ehingant@ubiobio.cl}} \and Romain~Yvinec\thanks{PRC, INRA, CNRS, IFCE, Université de Tours, 37380 Nouzilly, France. E-mail: \texttt{romain.yvinec@inrae.fr}}
\thanks{Associate researcher at Cogitamus laboratory.}
}

\date{\small \today}

\begin{document}

\maketitle

\begin{abstract}We study a stochastic version of the classical Becker-D\"oring model, a well-known kinetic model for cluster formation that predicts the existence of a long-lived metastable state before a thermodynamically unfavorable nucleation occurs, leading to a phase transition phenomena. This continuous-time Markov chain model has received little attention, compared to its deterministic differential equations counterpart. We show that the stochastic formulation leads to a precise and quantitative description of stochastic nucleation events thanks to an exponentially ergodic quasi-stationary distribution for the process conditionally on nucleation  has not yet occurred.

\smallskip

\noindent \footnotesize {\bfseries Keywords}: Quasi-stationary distribution; Metastability; Exponential ergodicity; Becker-D\"oring; Nucleation; Phase transition

\noindent \footnotesize {\bfseries AMS MSC 2010}: 82C26; 60J27.

\end{abstract}

\section{Introduction}

The Becker-Döring model is a kinetic model for phase transition phenomenon  represented schematically by the reaction network
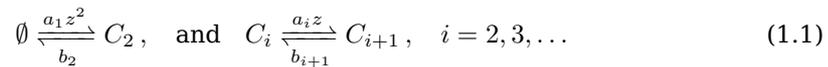
\begin{equation}\label{eq:chemreact_BD}
    \emptyset \ds \xrightleftharpoons[b_{2}]{a_1 z^2}  C_{2} \,, \quad \text{and} \quad C_i  \xrightleftharpoons[b_{i+1}]{a_{i}z}  C_{i+1} \,, \quad   i=2,3,\ldots
\end{equation}
We assume an infinite reservoir of monomer, cluster of size $1$, represented in \eqref{eq:chemreact_BD} by $\emptyset$. The parameter $z$ represents the \textit{fixed} concentration of monomer and will play a key role in the sequel. A cluster of size  $i\geq 2$, whose population is represented in \eqref{eq:chemreact_BD} by $C_i$, lengthen to give rise to a cluster of size $i+1$ at rate $a_iz$ or shorten to give rise a cluster of size $i-1$ at rate $b_i$. The rate of apparition of a new cluster of size $2$ is $a_1z^2$ (without loss of generality). All parameters are positives.

The Becker-D\"oring (BD) model goes back to the seminal work \emph{``Kinetic treatment of nucleation in supersaturated vapors''}~in  \cite{Becker1935}.
Since then, the model met very different applications ranging from physics to biology. From the mathematical point of view, this model received much more attention in the deterministic literature than the probabilistic one. We refer to our review \cite{Hingant2017} for historical comments and detailed literature review on theoretical results from the deterministic side. See also \cite{Hingant2019,Sun2018} for recent results
on functional law of large number and central limit theorem.
 
The model was initially designed to explain critical phase condensation phenomena where macroscopic droplets self-assemble and segregate from an initially supersaturated homogeneous mixture of particles, at a rate that is exponentially small in the excess of particles. This led to important applications in kinetic nucleation theory \cite{Schmelzer2005}. Mathematical studies in the 90's showed that (in the deterministic context), departing from certain initial conditions, the size distribution of clusters reaches quickly a metastable configuration composed of "small" clusters, and remains arbitrary close to that state for a very large time, before it converges to the true stationary solution that leads to "infinitely large" clusters (interpreted as droplets) \cite{Kreer1993,Penrose1989}. 

Our objective in this note is to re-visit the metastability theory in Becker-D\"oring model in terms of quasi-stationary distribution (QSD) for the associated continuous-time Markov chain. We prove existence, uniqueness and exponential ergodicity of a QSD for the BD model conditioned on 
the event that large clusters have not yet appeared. We prove furthermore that the convergence rate towards the QSD is faster than the rate of apparition of (sufficiently) large clusters. Quantitative results are obtained thanks to a surprisingly simple analytical formula for the QSD, that provides also an exact rate of apparition of stable large clusters, consistently with the original heuristic development of Becker and D\"oring.

\noindent \textbf{Outline}: We introduce in Sec. \ref{sec:model} the stochastic BD model. In Sec. \ref{sec:chainX}, we gather few (known) results on a related Birth-Death process, and from deterministic modeling of the BD model. Then, we prove exponential decay, in total variation, of the BD process towards its stationary measure in Sec. \ref{sec:long-time}. For any $n\geq 2$, similar results are obtained in Sec. \ref{sec:qsd} for the process conditionally on no cluster larger than $n$ are formed. An estimate on the time for the first cluster larger than $n$ to appear is provided in Sec. \ref{sec:taun}. The last, but not least, Sec. \ref{sec-final} gathers our results in a quantitative way, and proves that the QSD can be interpreted as a long-lived metastable state, when $n$ equals the critical nucleus size.

\noindent \textbf{Notation}: We denote by $\Nb_0$ the set of non-negatives integers $\Nb_0 = \{0,1,2\ldots\}$, $\Nb_{i}$ the set of integers greater or equal to $i$,  $[\![\ ]\!]$ for integers interval. For a set $A$, $A^c$ is its complement, $\mathbf 1_A$ the indicator function on it, $\# A$ its cardinality. For two sets $A$ and $B$ {their intersection is denoted by $A,\, B$}. Also, $\mathbf 1_x=\mathbf 1_{\{x\}}$ and $\mathbf 1$ (resp. $\mathbf 0$) is the constant function equal to $1$ (resp. $0$). For two numbers $a,b$, their minimum is $a\wedge b$. We denote by $\|\mu -\nu\|$ the total variation distance between two probability measure $\mu$ and $\nu$ on a countable state space $\mathcal S$ namely
\[\|\mu-\nu \| = \frac 1 2 \sum_{x\in\mathcal S} |\mu(x)-\mu(x)| = \inf_{\gamma\in\Gamma} \int_{S\times S} \mathbf 1_{x\neq y}\, \gamma(dx,dy) \,,\]
where $\Gamma$ is the set of probability measures on $\mathcal S\times \mathcal S$ with marginals $\mu$ and $\nu$. $\mathbf E$ (resp.  $\mathbf E_\mu$) denotes the expectation with respect to the usual probability measure $\mathbf P$ (resp. $\mu$).  

Below, we set $\Ec = \ell^1(\Nb_2,\Nb_0)$  the space of summable $\Nb_0$-valued sequences indexed by $\Nb_2$.

\section{The model}\label{sec:model}

The stochastic Becker-D\"oring (BD) process is a continuous-time Markov chain on the countable state space $\Ec$ with infinitesimal generator $\mathcal A $, given for all $\psi$ with finite support on $\Ec$ and $C\in \Ec$, by
\begin{equation*} 
    \Ac\psi(C) = \sum_{i=1}^{+\infty} \Big( a_i z C_i  [\psi(C + \Delta_i)-\psi(C)] + b_{i+1}  C_{i+1}[\psi(C- \Delta_{i})-\psi(C)] \Big)
\end{equation*}
with the convention $C_1=z$,  $\Delta_1=\mathbf e_2 $ and $\Delta_i =\mathbf e_{i+1} - \mathbf e_i$, for each $i\geq 2$, where $\{\mathbf e_2,\mathbf e_3,\ldots\}$ denotes the canonical basis of $\Ec$ namely, $e_{i,k}=1$ if $k=i$ and $0$ otherwise. 

We shall however use a different approach, modeling explicitly the size of each individual cluster. On a sufficiently large probability space $(\Omega,\mathcal F,\mathbf P)$, we introduce:
\begin{itemize}[wide,nosep,labelindent=0pt]
\item $N_1,N_2,\ldots$ a denumerable family of independent Poisson point measure with intensity the Lebesgue measure $dsdu$ on $\Rb_+^2$.
\item  $T_1,T_2,\ldots$ a collection of random times such that the $T_{k}-T_{k-1}$ are independent exponential random variable of parameter $a_1z^2$,  independent from the above Poisson point measure as well, with $T_0=0$.
\end{itemize}
Let $\Pi^\init$ a probability distribution on $\Ec$ and $\mathbf{C}(0)=(C_2(0),C_3(0),\ldots)$ an $\Ec$-valued random variable distributed according to $\Pi^\init$. We denote by $N^\init=\sum_{i=2}^\infty C_i(0)$. By construction $N^\init < \infty$ almost surely (\emph{a.s.}). Then, given $\mathbf{C}(0)$, we define $X_1(0),X_2(0),\ldots$ a denumerable collection of random variables on $\Nb_2$ such that, \emph{a.s.} for each $i\geq 2$,
\begin{equation}\label{eq:def-init-X_k}
C_i(0) = \# \set{k\in[\![1,N^\init]\!]}{X_k(0)=i}\,,
\end{equation}
and $X_k^\init = 2$ for all  $k > N^\init$. Note this construction may be achieved by a bijective \emph{labeling} function\footnote{A function (that exists) which associates, to each $c\in\Ec$ such that \smash{$N=\sum_{i=2}^\infty c_i<\infty$}, a unique sequence $(x_1,\ldots,x_N)$ in $\Nb_2$ satisfying $c_i=\#\set{k\in[\![1,N]\!]}{x_k=i}$.}.   Finally, we consider the denumerable collection of stochastic processes $X_1,X_2,\ldots$ on $\Nb_1$ solution of the stochastic differential equations, for all $t\geq 0$ and $k\geq 1$, 
\begin{equation}\label{eq:dif-eq-Xi}
X_k(t)  =   X_k(0) + \sum_{i=2}^\infty \int_0^t \int_{\Rb^+} \mathbf{1}_{s>T_{k-N^\init}}  \mathbf{1}_{X_k(s^-)=i} \big(\mathbf{1}_{u\leq a_i z} - \mathbf{1}_{a_iz <  u \leq  a_iz + b_i} \big)N_{k}(ds,du)\,,
\end{equation}
where by convention $T_k = 0$ if $k\leq 0$. The pathwise construction \eqref{eq:dif-eq-Xi} is what we call thereafter the \emph{particle description of the BD process}. The interpretation is clear: $X_k(t)$ denotes the size of the cluster labelled by $k$ at time $t$; for $k\leq N^\init$, clusters are initially "actives" while for $k>N^\init$ clusters are initially ``inactives'' at state $2$, and become ``activated'' at the random arrival times $T_{k-N^\init}$.  
We ensured $X_k(0)<\infty$ \emph{a.s.} because the $C_i(0)$'s are integer-valued random variables and belong to $\Ec$, the sequence $\mathbf C(0)$ is \emph{a.s.} equally $0$ from a certain range. Thus, local existence of \emph{c\`adl\`ag} processes $t\mapsto X_k(t)$ on $\Nb_1$ solution to \eqref{eq:dif-eq-Xi} can classically be obtained inductively. It is clear from \eqref{eq:dif-eq-Xi} that each $X_k$ evolves like a Birth-Death process for $t>T_k$ (that will be detailed in the next section \ref{sec:chainX}) and are mutually independent conditionally to their initial value. The Reuter's criterion gives a well-known necessary and sufficient condition so that each process $X_k$ is non-explosive, namely 
\begin{equation}\label{eq:ass_nonexplosive_birthdeath} \tag{H0}
    \sum_{n=2}^{\infty}Q_nz^n\left( \sum_{k=n}^{\infty} \frac{1}{a_kQ_k z^{k+1}}\right) =\infty\,, \ \text{with} \ Q_1=1\,, \ Q_i=\frac{a_1 a_2\cdots a_{i-1}}{b_2\cdots b_{i}}\,, \ i\geq 2\,.
\end{equation}
It is now convenient to go back to the original description at stake. The number of "active" clusters at time $t\geq 0$ is given by the counting process 
\begin{equation*}
N(t)= N^\init+ \sum_{k\geq 1} \indic{t\geq T_k} \,,
\end{equation*}
while the number of cluster of size $i\geq 2$ is
\begin{equation*}
C_i(t) = \#\set{k\in[\![1,N(t)]\!]}{X_k(t)=i}\,.
\end{equation*}
Now, noticing that $C_i(t)=\sum_{k=1}^{N(t)} \mathbf 1_{i} (X_k(t))$, we can prove from standard stochastic calculus that the process $\mathbf C$ given by $\mathbf C (t)=\left(C_2(t),C_3(t),\ldots\right)$ for all $t \geq 0$ has infinitesimal generator $\mathcal A$, and being non-explosive under condition  \eqref{eq:ass_nonexplosive_birthdeath}, it is the unique regular jump homogeneous Markov chain on $\Ec$ with infinitesimal generator $\mathcal A$ and initial distribution $\Pi^\init$, say the BD process. The proof is left to the reader and follows from classical theory, \emph{e.g.} \cite{Anderson1991}. In the sequel, $\mathbf C(t)$ always denote a BD process, and $\probin{\Pi^\init}{\mathbf C \in \cdot}$ its (unique) finite dimensional probability distribution given that $\mathbf C(0)$ is distributed according to  $\Pi^\init$. {We also set by convention $\mathbf P_C = \mathbf P_{\delta_C}$ for a deterministic $C\in \Ec$ and we recover $\probin{\Pi^\init}{\cdot} = \sum_{C\in \Ec} \probin{C}{\cdot} \, \Pi^\init(C)$.}
\section{Behaviour of one cluster}\label{sec:chainX}

Let $X$ the continuous-time Markov chain on $\Nb_1$ with transition rate matrix $(q_{i,j})_{i,j\geq 1}$ whose nonzero entries are 
\begin{equation}\label{eq:marche_aleatoire_absorbe}
q_{i,i+1}  = a_iz\,,\  q_{i,i-1} = b_{i}\,, \ q_{i,i}  =  -(a_iz+b_i)\,, \qquad i\geq 2\,.
\end{equation}
Remark that $i=1$ is absorbing in agreement with \eqref{eq:chemreact_BD}: when a cluster size reaches $1$, it ``leaves the system''. 
We shall assume standard hypotheses in the BD model \cite{Kreer1993,Penrose1989}:
\begin{equation}\label{hyp:H1}\tag{H1}
    \lim_{i\to \infty} b_i/a_i = z_s>0 \,, \quad \text{and} \quad \lim_{i\to \infty} b_{i+1}/b_i = 1\,.
\end{equation}
Hypothesis \eqref{hyp:H1} then guarantee \eqref{eq:ass_nonexplosive_birthdeath} for $z\neq z_s$ for the following reason. The convergence of both series 
\begin{equation*}
    \sum_{k=2}^\infty Q_kz^k \quad \text{and} \quad \sum_{k=2}^\infty \frac{1}{a_kQ_kz^{k+1}}\,,
\end{equation*}
depends on the value of $z$. Indeed, $z_s$ is the radius of convergence of the first series while the second series converges for $z>z_s$ and diverges for $z<z_s$.  We have a dichotomy in the long time behavior of $X$ related to this value. The case $z<z_s$ is  called the \emph{sub-critical case}, for which absorption at state $1$ is certain and the expected time of absorption is finite (also called ergodic absorption). The case $z>z_s$ is called the \emph{super-critical case} and absorption at $1$ is not certain (also called transient absorption), and the probability to be absorbed at $1$ is, according to \cite{Karlin1957},
\begin{equation}\label{eq:def-J}
    \lim_{t\to\infty} p_{i1}(t) = J\sum_{k=i}^{\infty} \frac{1}{a_kQ_kz^{k+1}}\,, \qquad \text{with} \ J=\left(\sum_{k=1}^{\infty} \frac{1}{a_kQ_kz^{k+1}}\right)^{-1}
\end{equation}
where $p_{ij}(t) = \probc{X(t)=j}{X(0)=i}$ the probability transition function of $X$. The limit case $z=z_s$ is somewhat technical and depends more deeply on the shape of the coefficients. It is not considered in this note. 

Following \cite{Kreer1993}, a precise long time estimate on transient states can be obtained, under the hypothesis \eqref{hyp:H1} and
\begin{equation}\label{hyp:H2}\tag{H2}
    \frac{b_{i+1}}{b_i}-1 = O(i^{-1})\,, \quad \frac{a_{i+1}}{a_i}-1 = O(i^{-1})\,, \quad  a_i =O(i)\quad \text{and}\quad \lim_{i\to+\infty} a_i = +\infty\,.
\end{equation}
In such a case, the infinite matrix $(q_{i,j})_{i,j\geq 2}$ in \eqref{eq:marche_aleatoire_absorbe} is self-adjoint on the Hilbert space $H$ consisting of the real sequences $\mathbf x=(x_2,x_3,\ldots)$ whose norm is
\begin{equation*}
\|\mathbf x\|_H = \sqrt{\sum_{i=2}^\infty \frac{x_i^2}{Q_iz^i}} <\infty\,.
\end{equation*}
We denote by $\langle \cdot, \cdot \rangle_H$ the associated scalar product. 
It turns that $(q_{i,j})_{i,j\geq 2}$ has a negative maximum eigenvalue $-\lambda$, and the following estimate holds for any $i\geq 2$, 
\begin{equation} \label{eq:cv-transitions}
    \|(p_{ij}(t))_{j\geq 2}\|_H  \leq  e^{-\lambda t}\|(p_{ij}(0))_{j\geq 2}\|_H =\frac{e^{-\lambda t}}{\sqrt{Q_iz^i}}\,.
\end{equation}
We will also consider the chain $X$ conditioned to remain below a given state $n+1\geq 2$. We define the exit time 
\begin{equation}\label{eq:def_Tn}
T_n=\inf\left(t\geq 0 \mid X(t)\notin [\![1,n]\!] \right)=\inf\left(t\geq 0 \mid X(t)\geq n+1\right)\,.
\end{equation}
Let $Y$ the birth-death process defined by $Y(t)=X(t \wedge T_n)$. Hence, $Y$ is absorbed either in $1$ or $n+1$, and the probability to be absorbed at $1$ (without visiting state $n+1$) is, according to \cite[p.387]{Karlin1957},
\begin{equation}\label{eq:def-Jn}
    \lim_{t\to +\infty} p_{i1}^n(t) = J_{n} \sum_{k=i}^{n} \frac{1}{a_kQ_kz^{k+1}} \,,\qquad \text{with} \ J_{n} =  \left(\sum_{k=1}^{n} \frac{1}{a_kQ_kz^{k+1}}\right)^{-1}\,,
\end{equation}
where $p_{ij}^n(t)=\mathbf P \left(Y(t) = j \mid Y(0)=i\right)$ is the probability transition function of $Y$ and clearly
\begin{equation}\label{eq:minorboundsTn}
\probc{T_n>t}{X(0)=i}\geq \lim_{t\to + \infty}\probc{T_n>t}{X(0)=i}=J_{n} \sum_{k=i}^{n} \frac{1}{a_kQ_kz^{k+1}}\,.
\end{equation}
Again, in \cite{Kreer1993}, the author shows that the truncated matrix $(q_{i,j})_{i,j=2,\ldots,n}$ is similar to a symmetric one and then there exists $\gamma_n>0$ such that for each $i=2,\ldots,n$,
\begin{equation}\label{eq:estim-pij}
    \sqrt{\sum_{j=2}^{n} \frac{p_{ij}^n(t)^2}{Q_jz^j} } \leq \frac{e^{-\gamma_n t}}{\sqrt{Q_iz^i}}\,.
\end{equation}
Note the probability to be absorbed in $1$ before time $t$, $p_{i1}^n(t)$, is monotonously increasing and $ \lim_{t\to +\infty} p_{i1}^n(t) = 1- \lim_{t\to+\infty} p_{i(n+1)}^n(t)$, thus we deduce that
\begin{equation*}
\sum_{j=2}^{n} \probc{Y(t) = j}{Y(0)=i\,,\ T_n>t}=\frac{\sum_{j=2}^{n} p_{ij}^n(t)}{1-p_{i(n+1)}^n(t)} \leq M_{i,n} e^{-\gamma_n t} 
\end{equation*}
where the constant $M_{i,n}$, obtained thanks to \eqref{eq:estim-pij}, Cauchy-Schwarz inequality and \eqref{eq:def-Jn}, is given by
\begin{equation*}
 M_{i,n} = \left(\frac{1}{Q_iz^i} \sum_{j=2}^{n} Q_jz^j  \right)^{\tfrac 1 2}\frac{1}{J_{n} \sum_{k=i}^{n} \frac{1}{a_kQ_kz^{k+1}}} \,.
\end{equation*}
We end this preliminary section, noticing that $X(t)=Y(t)$ on $\{T_n>t\}$, with 
\begin{equation}\label{eq:estime_absorbtion_kreer_n}
\probc{X(t)=1}{X(0)=i\,, \ T_n>t}  \geq 1- M_{i,n} e^{-\gamma_n t}\wedge 1 \,.
\end{equation}
\section{Long-time behaviour of the BD process} \label{sec:long-time}

In this section we are concerned with the long-time behaviour of the BD process. Formally the measure $\Pi^\eq$, given by
\[\Pi^\eq(C)= \prod_{i=2}^\infty e^{-c_i^\eq} \frac{(c_i^\eq)^{C_i} }{C_i!}\,, \qquad \text{with} \ c_i^\eq = Q_iz^i\]
for all $C\in\Ec$, satisfies  $\mathbf E_{\Pi^\eq}[\mathcal A\psi(C)]=0$ for any function $\psi$ on $\Ec$ with finite support\footnote{{In the sequel $C$ in expectation formula always refers to the free variable of integration.}}. Actually, $\Pi^\eq$ satisfies the detailed balance condition
$a_i z C_i \Pi^\eq(C)=b_{i+1}(C_{i+1}+1)\Pi^\eq(C+\Delta_i)$, 
for all $i\geq 1$ and all $C\in \Ec$ (with the convention that $C_1=z$), as a consequence of the relation $a_iQ_i = b_{i+1}Q_{i+1}$. In the sub-critical case, $\Pi^\eq$ is a probability measure on $\mathcal E$ (indeed $\Pi^\eq(\Nb_0^{\Nb_2})=1$ with support in $\Ec$ because of $\espin{\Pi^\eq}{\sum_{i=2}^\infty C_i} = \sum_{i=2}^\infty c_i^\eq <\infty$) and we prove exponential ergodicity towards $\Pi^\eq$. In the super-critical case, $\Pi^\eq$ is not a limiting distribution (and $\sum_{i=2}^\infty c_i^\eq =\infty$) but the measure defined by
\begin{equation}\label{eq:def-fi}
\Pi^\stat(C) = \prod_{i=2}^\infty e^{-f_i} \frac{(f_i)^{C_i} }{C_i!}\,, \qquad \text{with } f_i = JQ_iz^i \sum_{k=i}^\infty \frac{1}{a_kQ_kz^k}
\end{equation}
for all $C\in\Ec$, where $J$ is given in \eqref{eq:def-J}, characterizes the long-time behaviour of any finite-number of marginals. Now on, we note $\mathbf f = (f_i)_{i\geq 2}$, with $f_i$ defined in \eqref{eq:def-fi}.
\begin{theorem}\label{thm:limit-distribution} Under hypotheses \eqref{hyp:H1} and \eqref{hyp:H2}. Let $\Pi^\init$ a probability distribution on $\mathcal E$ such that
\begin{equation}\label{hyp:esp-Cin}
\espin{\Pi^\init}{\langle C,\sqrt{\mathbf Q}\rangle_H}<\infty
\end{equation} 
where $\sqrt{\mathbf Q} = (\sqrt{Q_iz^i})_{i\geq 2}$. With $\lambda>0$ introduced in Sec. \ref{sec:chainX}, see \eqref{eq:cv-transitions}, we have:
\begin{itemize}
\item In the sub-critical case ($z<z_s$), 
for all $t\geq 0$, 
    \[\|\probin{\Pi^\init}{\mathbf C(t) \in \cdot} - \Pi^\eq \| \leq  R^\init e^{-\lambda t}\]
    with  $R^\init = K(\espin{\Pi^\init}{\langle C,\sqrt{\mathbf Q}\rangle_H} + \espin{\Pi^\eq}{\Brak{C}{\sqrt{\mathbf Q}}_H})$ and  $K=(\sum_{k=2}^\infty c_i^\eq)^{\tfrac 1 2}$;
\item In the super-critical case ($z>z_s$), for all $t\geq 0$, and for all $n\geq 2$,
    \begin{equation*}
     \|\probin{\Pi^\init}{(C_2(t),\ldots, C_n(t)) \in \cdot} -  \Pi^\stat(\cdot\times \prod_{k=n+1}^\infty \Nb_0) \|      \leq  R^\init_n e^{-\lambda t} 
     \end{equation*}
     with $R^\init_n=K_n \espin{\Pi^\init}{\langle C,\sqrt{\mathbf Q}\rangle_H} + \|\mathbf f\|_H$ 
     and $K_n=(\sum_{k=2}^n c_i^\eq)^{\tfrac 1 2}$.
\end{itemize}
\end{theorem}
Not least, remark that $\espin{\Pi^\eq}{\Brak{C}{\sqrt{\mathbf Q}}_H} = \sum_{i=2}^\infty \sqrt{Q_iz^i} <\infty$ for $z<z_s$ and that $\mathbf f\in H$ for $z>z_s$ (see \cite{Kreer1993}). In the sequel, $K_n$ and $K$ always refer to the constants given above. The proof is based on a coupling (described below) to a distribution starting from $\mathbf 0$, so that the control of the initial particles in $\mathbf C(0)$ are a key point.

\begin{lemma}\label{lem:large-time-init-cluster}
  Under the hypothesis of Theorem \ref{thm:limit-distribution}. Let the collection of processes $X_1,X_2,\ldots$ being a particle description of the BD process $\mathbf C$. We have, for each $n\geq 2$, 
 \[\probin{\Pi^\init}{\forall k\in [\![ 1,N^\init ]\!]\,,\ X_k(t)\notin [\![2,n]\!]} \geq 1 -  K_n   \espin{\Pi^\init}{\langle  C,\sqrt{\mathbf Q}\rangle_H} e^{-\lambda t}\,.\] 
 In particular, for the sub-critical case,
 \[\probin{\Pi^\init}{\forall k\in [\![ 1,N^\init ]\!]\,,\ X_k(t)=1} \geq 1 -  K   \espin{\Pi^\init}{\langle C,\sqrt{\mathbf Q}\rangle_H} e^{-\lambda t}\,.\]
\end{lemma}

\begin{proof}
    Fix $n\geq 2$. Let $C\in \Ec$ deterministic, define $N=\sum_{i=2}^\infty C_i$ and $(i_1,\ldots,i_N) \in \Nb_2^N$ given by the \emph{labelling function}, \textit{e.g.} $\mathbf C(0)=C$ and $X_1(0)=i_1,\ldots X_N(0)=i_N$ satisfy relation \eqref{eq:def-init-X_k}. Since each processes $X_1,\ldots X_N$ are independent copy of the chain $X$ given in Sec. \ref{sec:chainX}, conditionally on their initial condition, we have
    \begin{equation} \label{eq:inter-abs-2n-1}
     \probin{C}{\forall k\in [\![ 1,N ]\!]\,,\ X_k(t)\notin [\![2,n]\!]} 
     = \prod_{k=1}^N \probc{X(t)\notin [\![2,n]\!]}{X(0)=i_k}
    \end{equation}
    for all $t\geq 0$. Thanks to Cauchy–Schwarz inequality and \eqref{eq:cv-transitions},
    \begin{equation*}\probc{X(t)\in [\![2,n]\!]}{X(0)=i} = \sum_{j=2}^n  p_{ij}(t) \leq \left(\frac{1}{\sqrt{Q_iz^i}}\right)K_n e^{-\lambda t}\wedge 1\,.
    \end{equation*}
    Hence, with \eqref{eq:inter-abs-2n-1}, we have
    \begin{equation*}\probin{C}{\forall k\in [\![ 1,N ]\!]\,,\ X_k(t)\notin [\![2,n]\!]} 
    \geq 1  -  K_ne^{-\lambda t} \sum_{k=1}^N \frac{1}{\sqrt{Q_{i_k} z^{i_k}}}\,,
    \end{equation*}
	remarking that $\prod_{i=1}^N (1-x_i\wedge 1) \geq 1 - \sum_{i=1}^N x_i$ for any non-negatives $x_1,\ldots, x_N$. 
	Finally, we conclude that 
	\begin{multline*}
	 \probin{\Pi^\init}{\forall k\in [\![ 1,N^\init ]\!]\,,\ X_k(t)\notin [\![2,n]\!]} 
     \geq 1 -  K_n e^{-\lambda t}  \sum_{C\in\Ec} \sum_{k=1}^N \frac{1}{\sqrt{Q_{i_k} z^{i_k}}} \Pi^\init(C)\\
    = 1 -  K_ne^{-\lambda t}\espin{\Pi^\init}{\sum_{i=2}^\infty \frac{\#\set{k\in [\![ 1,N^\init ]\!]}{X_k(0) = i}}{\sqrt{Q_iz^i}}}\,,
    \end{multline*}
    and the proof ends.
\end{proof}
We now show, by a coupling argument, that any solution satisfying \eqref{hyp:esp-Cin} is in total variation exponentially close to the solution that starts with no cluster, namely the deterministic initial condition at $\mathbf 0$.
\begin{lemma}\label{lem:coupling-large-time}
Under the hypothesis of Theorem \ref{thm:limit-distribution}. For all $t\geq 0$, we have
\begin{itemize}[wide]
\item In the subcritical case ($z<z_s$),
 \begin{equation*}
  \|\probin{\Pi^\init}{\mathbf C(t) \in \cdot} - \probin{\mathbf 0}{\mathbf C(t)\in\cdot} \|\leq  K \espin{\Pi^\init}{\langle  C,\sqrt{\mathbf Q}\rangle_H} e^{-\lambda t}\,;
 \end{equation*}
 \item In the super-critical case ($z>z_s$), for all $n\geq 2$,
 \begin{equation*}
  \|\probin{\Pi^\init}{ (C_2(t),\ldots,C_n(t)) \in \cdot} - \probin{\mathbf 0}{(C_2(t),\ldots,C_n(t))\in\cdot} \| 
  \leq  K_n  \espin{\Pi^\init}{\langle C,\sqrt{\mathbf Q}\rangle_H} e^{-\lambda t}
 \end{equation*}
\end{itemize}
\end{lemma}
\begin{proof}
Let the collection of processes $X_1,X_2,\ldots$ (resp. $Y_1, Y_2,\ldots$) being a \emph{particle description} of the BD process that starts from the initial distribution $\Pi^\init$ (resp. from $\delta_{\mathbf 0}$). We couple the processes $X_1,X_2,\ldots$ to the processes  $Y_1,Y_2,\ldots$ such that all new particle ``activates'' simultaneously and evolves with the same jumps. Namely, $Y_k(t)=X_{k+N^\init}(t)$ for all $k\geq 1$ and all $t\geq 0$, {where $N^\init$ is distributed according to $\Pi^\init$.}

In the sub-critical case the proof readily follows from Lemma \ref{lem:large-time-init-cluster} and the definition of the total variation since
\begin{equation*}
	\|\probin{\Pi^\init}{\mathbf C(t) \in \cdot} - \probin{\mathbf 0}{\mathbf C(t)\in\cdot} \|
	 \leq  1 - \probin{\Pi^\init}{\forall k\in[\![1,N^\init]\!], X_k(t)=1}
\end{equation*}
because all "active" clusters are equal whenever all initial clusters from $\Pi^\init$ have been absorbed. A very similar argument holds in the super-critical case.
\end{proof}

\begin{proof}[Proof of Theorem \ref{thm:limit-distribution}]
We consider first the sub-critical case. As said, condition \eqref{hyp:esp-Cin} is satisfied for $\Pi^\eq$ since $ \espin{\Pi^\eq}{\langle C,\sqrt{\mathbf Q}\rangle_H}<\infty$, thus Lemma  \ref{lem:coupling-large-time} applies for $\Pi^\eq$ as initial distribution.
Because the constructed BD process is regular and $\Pi^\eq$ is a stationary distribution {(\emph{i.e.} $\probin{\Pi^\eq}{\mathbf C(t)\in\cdot}=\Pi^\eq$)}, we deduce
\[\|\probin{\mathbf 0}{\mathbf C(t)\in\cdot} - \Pi^\eq\| \leq K  \espin{\Pi^\eq}{\langle C,\sqrt{\mathbf Q}\rangle_H} e^{-\lambda t}\,.  \]
Going back to any $\Pi^\init$ satisfying condition \eqref{hyp:esp-Cin}, applying Lemma \ref{lem:coupling-large-time} again and the triangular inequality yield the desired result.
 
\medskip

Consider now the super-critical case. $\Pi^\stat$ is a product of Poisson distribution $\mathcal P(f_i)$ on $\Nb_0$ with mean $f_i$. According to a classical result on Markov population processes, see e.g. \cite[Sec. 4]{Kingman1969}, the law $\probin{\mathbf 0}{\mathbf C(t)\in\cdot}$ is also a product of Poisson distribution $\mathcal P(c_i(t))$ on $\Nb_0$ with mean $c_i(t)$ such that $\mathbf c(t)=(c_2(t),c_3(t),\ldots)$ solves the deterministic (linear) Becker-Döring equations namely, $\mathbf c(t)=A \mathbf c(t)+a_1z^2 \mathbf e_2$ where $A=(q_{j,i})_{i,j\geq 2}$ the matrix with entries in \eqref{eq:marche_aleatoire_absorbe}, and initial condition $\mathbf c(0)=\mathbf 0$. Thanks to \cite[Theorem III]{Kreer1993}, we have
\begin{equation}\label{eq:Kreer_asymptot}
 \| \mathbf c(t) - \mathbf f\|_H \leq \|\mathbf f\|_H e^{-\lambda t} 
\end{equation}
and according to \cite[Lemma 1]{Ruzankin2004}, $\| \mathcal P(c_i(t)) - \mathcal P(f_i)\| \leq |c_i(t) - f_i|$. The latter, with independence of the marginals of $\Pi^\stat$ and of $\probin{\mathbf 0}{\mathbf C(t)\in\cdot}$, estimate \eqref{eq:Kreer_asymptot} and Cauchy-Schwarz inequality, entail
\[ \|\probin{\mathbf 0}{ (C_2(t),\ldots,C_n(t)) \in \cdot} - \Pi^\stat(\cdot\times\prod_{k=n+1}^\infty \Nb_0) \| \leq \sum_{i=2}^n |d_i(t) - f_i| \leq K_n \|\mathbf f\|_H e^{-\lambda t} \,.\]
We conclude again by Lemma \ref{lem:coupling-large-time} and the triangular inequality.
\end{proof}

\section{A quasi-stationary distribution}  \label{sec:qsd}

Let $\Ec_n=\set{C\in \Ec}{ C_i = 0\,, i\geq n+1}$. We define the first exit time from $\Ec_n$, 
\begin{equation}\label{eq:def-tau_n}
  \tau_n = \inf\set{t\geq 0}{\mathbf C(t)\notin \Ec_n}\,.
 \end{equation}
 Remark that $\probin{\Pi^\init}{\tau_n>t}>0$ for all times $t>0$ and for any $\Pi^\init$ supported on $\Ec_n$. We will give in the next section \ref{sec:taun} a tight lower bound on that probability in the super-critical case. In this section, we prove exponential ergodicity towards a unique QSD for the BD process conditioned to $\tau_n>t$. 
 It is remarkable that we have at hand an explicit QSD, given, for all $C\in\Ec_n$, by
\begin{equation}\label{QSD_poisson}
\Pi_n^\qsd (C)= \prod_{i=2}^{n} \frac{(f_i^n)^{C_i}}{C_i!}e^{-f_i^n}\,, \qquad \text{with } f_i^n(z) = J_{n} Q_iz^i \sum_{k=i}^{n} \frac{1}{a_kQ_kz^{k+1}}
\end{equation}
for $i=2,\ldots,n$ and $J_n$  defined in \eqref{eq:def-Jn}.

\begin{proposition}\label{prop:QSD}
Under assumption \eqref{eq:ass_nonexplosive_birthdeath}. The distribution $\Pi_n^\qsd$ is a quasi-stationary distribution for the BD process conditioned to stay on $\Ec_n$ namely, 
 \[\probcin{\Pi_n^\qsd}{\mathbf C(t)\in \cdot}{t<\tau_n} = \Pi_n^\qsd\]
and moreover 
 \[ \probin{\Pi_n^\qsd}{t<\tau_n} = \exp\left(-J_n t\right)\,.\]
\end{proposition}

\begin{proof}
Recall assumption \eqref{eq:ass_nonexplosive_birthdeath} ensures the BD process is regular. Fix $n\geq 2$. Let the semi-group $P_t^n\psi(C)=\espin{C}{\psi(\mathbf{C}(t))\mathbf 1_{t<\tau_n}}$ for $t\geq 0$ (\emph{i.e.} $\mathbf C(0)$ is distributed according to $\delta_C$), whose generator is
 \begin{equation*}
 \Ac_n\psi(C) = \sum_{i= 1}^{n-1} \Big( a_i z C_i  [\psi(C + \Delta_i)-\psi(c)] + b_{i+1}  C_{i+1}[\psi(C- \Delta_i)-\psi(C)] \Big)
 -a_{n}zC_{n}\psi(C)
\end{equation*}
for all $C\in\Ec_n$ (recall $C_1=z$) and bounded function $\psi$ on $\Ec_n$. Denote by  $\Ac_n^*$ the dual operator for the generator $\Ac_n$. Some calculations show that the distribution \eqref{QSD_poisson} satisfies, for any $C\in \Ec_n$,
\begin{equation*}
\Ac_n^*\Pi_n^\qsd (C)= \Pi_n^\qsd (C)\left\{ b_2f_2^n -a_1z^2 + \sum_{i=2}^{n}\frac{C_i}{f_i^n}\left( a_{i-1}zf_{i-1}^n-(a_iz + b_i)f_i^n + \mathbf 1_{i < n} b_{i+1}f_{i+1}^n\right)\right\}
\end{equation*}
with the convention $f_1^n=z$. Since the $f_i^n$ given by \eqref{QSD_poisson} verifies $a_izf_i^n-\mathbf 1_{i < n} b_{i+1}f_{i+1}^n=J_n$ for all $i\in [\![1,n ]\!]$, all terms but the first cancel in the above expression, so that we obtain $\Ac_n^* \Pi_n^\qsd = - J_n \Pi_n^\qsd$ which is the classical spectral criteria of QSD, noticing that $J_n\leq a_1z^2$, see \cite[Thm 4.4]{Collet2013}. 
\end{proof}

The next theorem shows the QSD is a quasi-limiting distribution for a wide range of initial distribution supported on $\Ec_n$, with an exponential rate of convergence and an explicit (non-uniform) pre-factor.

\begin{theorem}\label{thm:ergodicity-qsd}
Under assumption \eqref{eq:ass_nonexplosive_birthdeath}.  Let $\Pi^\init$ a probability distribution on $\Ec_n$ such that $\espin{\Pi^\init}{\sum_{i=2}^\infty C_i}<\infty$. We have for all $t\geq 0$, 
\begin{equation*}
\|\probcin{\Pi^\init}{\mathbf C(t) \in \cdot}{\tau_n>t} - \Pi_n^\qsd\| \leq  K_n \left(\frac{H_n^\init}{\prob{\tau_n>t}} + e^{J_n t} H_n^\qsd \right)e^{-\gamma_n t}\,,
\end{equation*} 
where $\tau_n$ is defined in \eqref{eq:def-tau_n}, $J_n$ in \eqref{eq:def-Jn}, $K_n$ in Theorem~\eqref{thm:limit-distribution}, $\gamma_n$ in \eqref{eq:estim-pij},
\begin{equation*}
H_n^\init = \sum_{i=2}^n \sqrt{Q_i z^i}\frac{ \espin{\Pi^\init}{C_i} }{  f_i^n  } \quad \text{and} \quad H_n^\qsd  = \sum_{i=2}^n \sqrt{Q_i z^i} \,.
\end{equation*}
\end{theorem}
It is clear that $H_n^\init$ is finite because $\espin{\Pi^\init}{\sum_{i=2}^\infty C_i}$ is. The proof of Theorem \ref{thm:ergodicity-qsd} is similar to the proof of Theorem \ref{thm:limit-distribution} and consists in a coupling argument together with a control of the initial clusters in $\Pi^\init$. We start by the later, which is the analogous of Lemma \ref{lem:large-time-init-cluster}.

\begin{lemma}\label{lem:At}
 Under the hypothesis of Theorem \ref{thm:ergodicity-qsd}. Let the collection of processes $X_1,X_2,\ldots$ being a particle description of the BD process $\mathbf C$.	We have 
	\begin{equation*}
	\probcin{\Pi^\init}{\forall k\in[\![1,N^\init]\!]\,,\ X_k(t)=1}{\tau_n>t}	\geq 1-  e^{-\gamma_n t}\frac{K_nH_n^\init}{\probin{\Pi^\init}{t<\tau_n}}\,.
	\end{equation*}
\end{lemma}

\begin{proof}
	We start by observing that the following relation holds true,
	\begin{equation}\label{eq:rel-tau-tau0}
	\tau_n = \min(\tau^0_n,T_n^1,\ldots,T^{N^\init}_n) \,,
	\end{equation}
	where $T^k_n=\inf\set{t>0}{X_k(t)\geq n+1}$ for $k=1,\ldots,N^\init$ and 
	\begin{equation}\label{def:tau0}
	\tau^0_n =  \inf\set{t\geq 0}{\exists k > N^\init\,, X_k(t)\geq n+1}\,.
	\end{equation}
    Let $\mathbf C(0)=C\in \Ec_n$ deterministic, define $N=\sum_{i=2}^\infty C_i$ and $(i_1,\ldots,i_N) \in [\![2,n]\!]^N$ given by \emph{labelling function} such that  $C_i=\#\set{k\in[\![1,N]\!]}{i_k=i}$. Conditionally on their initial condition, all clusters $X_k(t)$ (starting at $i_k$) are independent from each other, thus the event
	\begin{equation}\label{eq:defAt}
	A_t = \{\forall k\in[\![1,N]\!]\,,\ X_k(t)=1\}
	\end{equation}
	is independent of $(X_k)_{k>N}$ and thus independent of $\tau^0_n$. Then,
	\[\probcin{C}{A_t}{\tau_n>t} = \probcin{C}{A_t}{\min(T_n^1,\ldots,T_n^{N})>t}\,.\]
	Still by independence of the clusters from each other, we claim that
	\begin{equation}\label{eq-old-lemmaAt}
	\probcin{C}{A_t}{\min(T_n^1,\ldots,T_n^{N})>t} = \prod_{k=1}^N \probc{X(t)=1}{X(0)=i_k\,,\
		{T_n}>t}\,
	\end{equation}
	{where $X$ is defined in Sec. \ref{sec:chainX} and $T_n$ in \eqref{eq:def_Tn}}. This equation is clear for $N=1$, and is easily proved by induction. We do it only for $N=2$, for the sake of simplicity. Let $i_1$, $i_2 \in[\![2,n]\!]$.	By definition of $A_t$,
	\[\probcin{C}{A_t}{\min(T_n^1,T^{2}_n)>t} = \frac{\probin{C}{X_1(t)=1\,,\ X_2(t)=1\,,\ T_n^1>t\,,\ T_n^2>t}}{\probin{C}{T_n^1>t\,,\ T_n^2>t}}\,. \]
	Then, by independence of the $X_1$ and $X_2$ conditionally on their initial condition we have
	\begin{multline*}
	\probcin{C}{A_t}{\min(T_n^1,T^{2}_n)>t} = \prod_{k=1}^2 \frac{\prob{X_k(t)=1\,,\  T_n^k>t \,,\ X_k(0)=i_k}}{\prob{T_n^k>t\,, \ X_k(0)=i_k}}\\
	= \prod_{k=1}^2 \probc{X(t)=1}{X(0)=i_k\,,\
		 T_n>t}\,.
	\end{multline*}
	since $X_1$ and $X_2$ are independent copy of $X$, which proves the desired result. Thus, going back to \eqref{eq-old-lemmaAt} and thanks to \eqref{eq:estime_absorbtion_kreer_n} we have
	\begin{equation*}
    \probcin{C}{A_t}{\tau_n>t} 
    \geq  \prod_{k=1}^N \left(1- M_{i_k,n} e^{-\gamma_n t}\wedge 1\right) \geq 1 - e^{-\gamma_n t} \sum_{k=1}^N M_{i_k,n}\,.
	\end{equation*}
	Finally, we obtain
	\begin{multline*}
	\probcin{\Pi^\init}{A_t}{\tau_n>t}
	= \sum_{C\in \Ec_n} \probcin{C}{A_t}{\tau_n>t} \frac{\probin{C}{\tau_n>t}}{\probin{\Pi^\init}{\tau_n>t}} \Pi^\init(C)\\
	\geq 1 -  e^{-\gamma_n t}\sum_{C\in\Ec_n} \sum_{k=1}^N M_{i_k,n} \frac{\probin{C}{\tau_n>t}}{\probin{\Pi^\init}{\tau_n>t}} \Pi^\init(C)
	\geq 1-  e^{-\gamma_n t}\frac{K_n H_n^\init}{\probin{\Pi^\init}{\tau_n>t}}\,,
	\end{multline*}
	{where in the second line, $N$ and $(i_k)_{k=1..N}$ are given by the \emph{labelling} function for each $C\in \Ec_n$}, and using that $\probin{C}{\tau_n>t}\leq 1$ in the last inequality. Remark the expression of $H^\init_n$ is obtained thanks to the definition of $M_{i,n}$ in \eqref{eq:estime_absorbtion_kreer_n} and $f_i^n$ in \eqref{QSD_poisson}. 
\end{proof}
\begin{proof}[Proof of the Theorem \ref{thm:ergodicity-qsd}]
Let the collection of processes $X_1,X_2,\ldots$ (resp. $Y_1, Y_2,\ldots$) being a \emph{particle description} of the BD process that starts from the initial distribution $\Pi^\init$ (resp. from $\delta_{\mathbf 0}$). We couple the processes $X_1,X_2,\ldots$ to the processes  $Y_1,Y_2,\ldots$ as in the proof of Theorem \ref{thm:ergodicity-qsd}, namely, $Y_k(t)=X_{k+N^\init}(t)$ for all $k\geq 1$ and all $t\geq 0$, where $N^\init$ is distributed according to $\Pi^\init$.%

To avoid notation confusion, we write $\tau_n^X$ and $\tau_n^Y$ the first exit time from $\Ec_n$ for the collection of processes $\{X_k\}$ and $\{Y_k\}$, respectively. We define $\tau_n^{0,X}$ and $\tau_n^{0,Y}$, respectively to the processes $\{X_k\}$ and $\{Y_k\}$ likewise $\tau^0_n$ in \eqref{def:tau0}.
Due to the coupling between the $\{X_k\}$ and $\{Y_k\}$, we have
\begin{multline}\label{eq:relation_taun_X_Y}
\tau_n^{0,Y}=\tau_n^{Y}=\inf\set{t\geq 0}{\exists k>0\,, Y_k(t)\geq n+1}\\=\inf\set{t\geq 0}{\exists k > N^\init\,, X_k(t)\geq n+1}=\tau_n^{0,X}\,.
\end{multline}
Remark that each $Y_k$, for $k\geq 1$, is independent of $X_i$ for $i\leq N^\init$, thus independent of $T_n^1,\ldots,T^{N^\init}_n$ {the exit times arising in \eqref{eq:rel-tau-tau0}}. Hence, by \eqref{eq:rel-tau-tau0} and \eqref{eq:relation_taun_X_Y}, {the laws of} the collection of processes $\{Y_k\}$ conditioned to $\tau_n^Y>t$  equals to {the laws of} the collection of processes $\{Y_k\}$ conditioned to $\tau_n^{X}>t$.
Also, we have, for any $i\geq 2$ and $t>0$, $\#\set{k}{X_k(t)=i} = \#\set{k}{Y_k(t)=i}$ on the event $A_t$, given in \eqref{eq:defAt}, since all initial particles being absorbed. Finally, we deduce that
\begin{multline*}
\|\probcin{\Pi^\init}{\mathbf C(t)\in\cdot}{\tau_n>t} - \probcin{\mathbf 0}{\mathbf C(t)\in\cdot}{\tau_n>t}\| \\
\leq \probc{\exists i\geq 2,\, \#\set{k}{X_k(t)=i} \neq \#\set{k}{Y_k(t)=i}}{\tau_n^X>t} 
\leq \probcin{\Pi^\init}{A_t^c}{t<\tau_n^X}\,.
\end{multline*}
The latter, with Lemma \ref{lem:At}, entails
\begin{equation}\label{eq:distance-C-D}
\|\probcin{\Pi^\init}{\mathbf C(t)\in\cdot}{\tau_n>t} - \probcin{\mathbf 0}{\mathbf C(t)\in\cdot}{\tau_n>t}\| \leq  e^{-\gamma_n t}\frac{K_n H_n^\init}{\probin{\Pi^\init}{t<\tau_n}}\,.
\end{equation}
Then, applying estimate \eqref{eq:distance-C-D} with the initial distribution $\Pi_n^\qsd$ since $\espin{\Pi^\qsd_n}{\sum_{i=2}^\infty C_i} = \sum_{i=2}^n f_i^n<\infty$ and by Proposition \ref{prop:QSD}, we deduce
\begin{equation*}
\|\Pi^\qsd_n - \probcin{\mathbf 0}{\mathbf C(t)\in\cdot}{t<\tau_n}\| \leq  e^{-\gamma_n t}K_n e^{J_n t} H_n^\qsd\,.
\end{equation*}
We end the proof by triangular inequality. 
\end{proof}

\section{Estimates on $\tau_n$ and the largest cluster}\label{sec:taun}

In this section we consider the super-critical case $z>z_s$. The analysis of the $\tau_n$ in \eqref{eq:def-tau_n} leads off the simple observation
\[\tau_n>t \Leftrightarrow \forall s\leq t, \max_{1\leq k\leq N(s)} X_k(s) \leq n\,, \]
where $X_1,X_2,\ldots$ is the particle description of the BD process. We prove 

\begin{theorem}\label{thm:tau_n}
Under assumption \eqref{eq:ass_nonexplosive_birthdeath}. Let $z>z_s$ and $\Pi^\init$ a probability distribution on $\Ec_n$ such that $\espin{\Pi^\init}{\sum_{i=2}^\infty C_i}<\infty$. We have
\[\probin{\Pi^\init}{\tau_n>t} \geq G_n^\init e^{-J_nt}\]
 where 
\[G_n^\init = \espin{\Pi^\init}{\prod_{i=1}^n \left(\frac{f_i^n}{Q_iz^i}\right)^{C_i}} \geq  1 - \sum_{i=2}^n \espin{\Pi^\init}{C_i}(1-\tfrac{f_i^n}{Q_iz^i}) \,.\]
\end{theorem}

In fact, as in the coupling strategy, \eqref{eq:rel-tau-tau0} provides a useful understanding of the statistics of $\tau_n$ by decomposing between the initial cluster from the ones that will appear at later times. We start with the statistics of the later, namely of $\tau_n^0$. The next lemma clearly applies for the initial distribution $\Pi^\init=\delta_{\mathbf 0}$ but is more general, and related to a maximum principle in the BD model \cite{Canizo2019}.
\begin{lemma}\label{lem:tau_couplage}
Under assumption \eqref{eq:ass_nonexplosive_birthdeath} and $z>z_s$. For any  probability distribution $\Pi^\init$ on $\Ec_n$ such that for all $i\in[\![2,n]\!]$ and $k\in \Nb$,
	\begin{equation}\label{hyp:domination}
	\probin{\Pi^\init}{\sum_{i\leq j \leq n}  C_j \geq k} \leq \probin{\Pi^\qsd}{\sum_{i\leq j \leq n}  C_j\geq k}\,,
	\end{equation}
we have
	\begin{equation*}
	\probin{\Pi^\init}{\tau_n>t}\geq e^{-J_n t}\,.
	\end{equation*}
\end{lemma}
\begin{proof}[Proof of Lemma \ref{lem:tau_couplage}]
It is classical that condition \eqref{hyp:domination} ensures there exists randoms $\mathbf C^\init$ and $\mathbf C^\qsd$ distributed according to $\Pi^\init$ and $\Pi^\qsd$, respectively, such that for each $i\in[\![2,n]\!]$,
\begin{equation*}
\sum_{i\leq j \leq n} C_j^\init \leq \sum_{i\leq j \leq n}  C^{\qsd}_j\,,\quad a.s.
\end{equation*}
see \emph{e.g.} \cite[Sec. 4.12]{Grimmett2001}. Then, we may construct the collection of processes $X_1,X_2,\ldots$ (resp. $Y_1, Y_2,\ldots$) as a \emph{particle description} of the BD process associated to $\mathbf C^\init$ (resp. to  $\mathbf C^\qsd$) such that, \emph{a.s.}, for all $i\geq 1$, $X_i(0)\leq Y_i(0)$. A standard coupling between two copies of the chain $X$ from Sec. \ref{sec:chainX} consists in having the same jumps in the two copies as soon as they are equal. Such coupling applied to each couple $(X_i,Y_i)$ then ensures that, for all $i\geq 1$ and $t\geq 0$, we have  $X_i(t)\leq Y_i(t)$ \emph{a.s.} In particular,
\[ \inf\{t>0\,\mid \,\max_{k} X_k(t)>n\} \geq \inf\{t>0\, \mid\, \max_{k} Y_k(t)>n\} \quad  a.s. \]
thus, with Proposition \ref{prop:QSD},
\[	\probin{\Pi^\init}{\tau_n>t}\geq 	\probin{\Pi^\qsd}{\tau_n>t}=e^{-J_nt}\,.\]
\end{proof}

\begin{proof}[Proof of Theorem \ref{thm:tau_n}]
Let $n\geq 2$ and $i \in[\![2,n]\!]$. Define $g_{n,i}(t) = \probc{T_n>t}{X(0)=i}$
where $T_n$ and $X$ are given in Sec. \ref{sec:chainX}. Thanks to \eqref{eq:minorboundsTn}, we have 
\begin{equation}\label{eq_erwan_KMcG}
g_{n,i}(t) \geq \lim_{t\to + \infty} g_{n,i}(t) = J_{n} \sum_{k=i}^{n} \frac{1}{a_kQ_kz^{k+1}} = \frac{f_i^n}{Q_iz^i} \,.
\end{equation}
Let $\psi_n(t,x)=1$ if $\max_{0\leq \tau \leq t} x(\tau) \leq n$ and $0$ otherwise. We have
\[\probin{\Pi^\init}{\tau_n>t}  = \sum_{C\in \Ec_n} \probin{C}{\prod_{i=1}^{N(t)} \psi_n(t,X_i)=1} \Pi^\init(C)\]
where $X_1,X_2,\ldots$ the particle description of the BD process $\mathbf C$ (starting at $\delta_C$). By independence of the particles conditionally to their initial condition,
\begin{multline}\label{eq:decompose_C_deter}
\probin{C}{\prod_{i=1}^{N(t)} \psi_n(t,X_i)=1} =  \probin{C}{\prod_{i=N^\init+1}^{N(t)} \psi_n(t,X_i)=1}\probin{C}{\prod_{i=1}^{N^\init} \psi_n(t,X_i)=1} \\ = \probin{\mathbf 0}{\tau_n>t} \prod_{k=1}^N g_{n,i_k}(t)
\end{multline}
where again, $N$ and $(i_k)_{k=1..N}$ are given by the \emph{labelling} function associated to $C\in \Ec_n$. Combining relations \eqref{eq_erwan_KMcG} and \eqref{eq:decompose_C_deter} with Lemma \ref{lem:tau_couplage} and summing over all initial conditions ends the proof.
\end{proof}
	
\section{Metastability close to z\textsubscript{s}}\label{sec-final}
In this section we assume additionally to \eqref{hyp:H1} and \eqref{hyp:H2}, to fit with \cite{Kreer1993,Penrose1989}, that
\begin{equation}\label{hyp:H3}\tag{H3}
A'< a_i< A i^\alpha \,,\quad \frac{b_{i+1}}{a_{i+1}} + \frac{\kappa}{i^\nu}\leq \frac{b_i}{a_i} \quad \text{and} \quad  z_s e^{Gi^{-\gamma}} \leq \frac{b_i}{a_i} \leq z_s e^{G'i^{-\gamma'}}\,,
\end{equation}
for all $i\geq 2$, where $\alpha,\,\gamma \in(0,1)$, $\gamma',\,\nu>0$, $\kappa$, $A'$, $A$, $G$ and $G'$ positives. We also use the terminology of \cite{Penrose1989} namely a quantity $q(z)$ of $z$ is: exponentially small if $q(z)/(z-z_s)^m$ is bounded for all $m>0$ as $z\!\ssearrow\!z_s$ ($z$ converges to $z_s$ and $z>z_s$); and  at most algebraically large if $(z-z_s)^{m_0}q(z)$ is bounded for some $m_0>0$ as $z\!\ssearrow\!z_s$.  

Assumption \eqref{hyp:H3} ensures the existence of a unique $n^*$ (depending on $z$) such that $b_{n^*+1}/a_{n^*+1} < z < b_{n^*}/a_{n^*}$. The size $n^*$ is interpreted as the nucleus size: for a cluster $X(t)\leq n^*$, $X(t)$ tends to shorten, while for $X(t)> n^*$, it tends to grow. With assumption \eqref{hyp:H3}, $n^*\to\infty$ as $z \ssearrow z_s$. In \cite{Kreer1993,Penrose1989}, $n^*$ is proved to be at most algebraically large. Moreover, the time scale $1/\gamma_{n^*}$ (in Theorem~\eqref{thm:ergodicity-qsd}) is also at most algebraically large, and $J_{n^*}$ (in Theorem~\ref{thm:tau_n}) is exponentially small. We now choose an initial distribution $\Pi^\init$ with support on $\Ec_j$, with $j$ independent of $z$ (no generality is claimed here). We have
\begin{equation}\label{eq:estimeeTVQSD1}
\probin{\Pi^\init}{\tau_{n^*}>t} \geq \left(1 - \sum_{i=2}^j \espin{\Pi^\init}{C_i}(1-\tfrac{f_i^{n^*}}{Q_iz^i})\right) e^{-J_{n^*}t}\,,
\end{equation} 
which is arbitrary close to one for times $t \ll 1/J_{n^*}$ as $f^{n^*}_i/(Q_iz^i)\to 1$ when $z\!\ssearrow\!z_s$. Moreover, we have
\begin{equation}\label{eq:estimeeTVQSD2}
\|\probcin{\Pi^\init}{\mathbf C(t) \in \cdot}{\tau_{n^*}>t} - \Pi_n^\qsd\| \leq (\tfrac{H_j^\init}{G^\init_j}+H_{n^*}^\qsd ) K_{n^*} e^{J_{n^*} t-\gamma_{n^*} t}
\end{equation}  
which is arbitrary small for times $1/\gamma_{n^*} \ll t \ll 1/J_{n^*}$. Indeed, note that $a_iQ_iz^i$ is decreasing up to the size $n^*$, thus $K_{n^*}^2  \leq \tfrac{a_1z}{A'}n^*$. Hence, $K_{n^*}$ as well as $H^\qsd_{n^*}$ are at most algebraically large. Equations \eqref{eq:estimeeTVQSD1}-\eqref{eq:estimeeTVQSD2} show the QSD is indeed a metastable state.

\providecommand{\bysame}{\leavevmode\hbox to3em{\hrulefill}\thinspace}
\providecommand{\MR}{\relax\ifhmode\unskip\space\fi MR }
\providecommand{\MRhref}[2]{%
  \href{http://www.ams.org/mathscinet-getitem?mr=#1}{#2}
}
\providecommand{\href}[2]{#2}

\bigskip

\noindent {\bfseries Acknowledgments.} E.~H. has to thanks the partial support by FONDECYT project n. 11170655.


\begin{thebibliography}{10}

\bibitem{Anderson1991}
W.~J. Anderson, \emph{Continuous-time {M}arkov chains}, Springer Series in
  Statistics: Probability and its Applications, Springer-Verlag, New York,
  1991, An applications-oriented approach. \MR{1118840}

\bibitem{Becker1935}
R.~Becker and W.~Döring, \emph{Kinetische behandlung der keimbildung in
  übersättigten dämpfen}, Annalen der Physik \textbf{416} (1935), no.~8,
  719--752.

\bibitem{Canizo2019}
J.~A.~Ca\~nizo, A.~Einav, B.~Lods, \emph{Uniform moment propagation for the
	Becker–Döring equations}, Proc.~R.~Soc.~Edinb, \textbf{149} (2019), 995--1015, 


\bibitem{Collet2013}
P.~Collet, S.~Mart\'{\i}nez, and J.~San~Mart\'{\i}n, \emph{Quasi-stationary
  distributions}, Probability and its Applications (New York), Springer,
  Heidelberg, 2013, Markov chains, diffusions and dynamical systems.
  \MR{2986807}

\bibitem{Grimmett2001}
G.~R. Grimmett and D.~R. Stirzaker, \emph{Probability and random
  processes}, third ed., Oxford University Press, New York, 2001. \MR{2059709}

\bibitem{Hingant2017}
E.~Hingant and R.~Yvinec, \emph{Deterministic and stochastic
  {B}ecker-{D}\"{o}ring equations: past and recent mathematical developments},
  Stochastic processes, multiscale modeling, and numerical methods for
  computational cellular biology, Springer, Cham, 2017, pp.~175--204.
  \MR{3726370}

\bibitem{Hingant2019}
\bysame, \emph{The {B}ecker-{D}\"{o}ring process: pathwise convergence and
  phase transition phenomena}, J. Stat. Phys. \textbf{177} (2019), no.~3,
  506--527. \MR{4026657}

\bibitem{Karlin1957}
S.~Karlin and J.~McGregor, \emph{The classification of birth and death
  processes}, Trans. Amer. Math. Soc. \textbf{86} (1957), 366--400. \MR{94854}

\bibitem{Kingman1969}
J.~F.~C. Kingman, \emph{Markov population processes}, J. Appl. Probability
  \textbf{6} (1969), 1--18. \MR{254934}

\bibitem{Kreer1993}
M.~Kreer, \emph{Classical {B}ecker-{D}\"{o}ring cluster equations: rigorous
  results on metastability and long-time behaviour}, Ann. Physik (8) \textbf{2}
  (1993), no.~4, 398--417. \MR{1223849}

\bibitem{Penrose1989}
O.~Penrose, \emph{Metastable states for the {B}ecker-{D}\"{o}ring cluster
  equations}, Comm. Math. Phys. \textbf{124} (1989), no.~4, 515--541.
  \MR{1014113}

\bibitem{Ruzankin2004}
P.~S. Ruzankin, \emph{On the rate of {P}oisson process approximation to a
  {B}ernoulli process}, J. Appl. Probab. \textbf{41} (2004), no.~1, 271--276.
  \MR{2036288}

\bibitem{Schmelzer2005}
J.~W.~P. Schmelzer (ed.), \emph{Nucleation theory and applications}, Wiley-VCH
  Verlag GmbH \& Co. KGaA, Weinheim, 2005.

\bibitem{Sun2018}
W.~Sun, \emph{A functional central limit theorem for the {B}ecker-{D}\"{o}ring
  model}, J. Stat. Phys. \textbf{171} (2018), no.~1, 145--165. \MR{3773855}

\end{thebibliography}
\end{document}